\documentclass[onecolumn,10pt]{article}
\usepackage[top=.75in, bottom=.75in, left=.75in, right=.75in]{geometry}
\usepackage{amsmath,amsfonts,amscd,amssymb,amsthm,bbm,bm,algorithm, algpseudocode}
\usepackage{graphicx}
\usepackage{epstopdf}
\usepackage{overpic}
\usepackage{cancel}
\usepackage{rotating}
\usepackage{url}
\usepackage{caption}
\usepackage{color}
\usepackage{rotating}
\usepackage{multirow}
\usepackage{wrapfig}
\usepackage{mathtools}
\usepackage{subeqnarray}
\usepackage{setspace}
\usepackage{palatino} % needs 10
\usepackage{microtype}
\usepackage[numbers,sort&compress]{natbib}

\usepackage[bottom,flushmargin,hang,multiple]{footmisc}
\usepackage{lipsum}
\newcommand\blfootnote[1]{%
  \begingroup
  \renewcommand\thefootnote{}\footnote{#1}%
  \addtocounter{footnote}{-1}%
  \endgroup
}

\definecolor{header1}{cmyk}{0,0,0,1}

\def \k {\mathbbm{k}}

\def \dim {\operatorname{dim}}

\def \GL {\operatorname{GL}}

\def \R {\mathbbm{R}}
\def \C {\mathbbm{C}}

\numberwithin{equation}{section}
\numberwithin{table}{section}
\numberwithin{equation}{section}
\newtheorem{theorem}{Theorem}[section]

\newtheorem{corollary}[theorem]{Corollary}
\newtheorem{definition}[theorem]{Definition}
\newtheorem{example}[theorem]{Example}
\newtheorem{remark}[theorem]{Remark}

\usepackage[utf8]{inputenc}

\usepackage[normalem]{ulem}
\usepackage{color}

\title{\vspace{-.125in}{\huge\selectfont \textbf{Simultaneous block diagonalization of \\ symmetric matrices via congruence}}\vspace{-.075in}}

\author{\normalsize{Lishan Fang$^{1}$, Hua-Lin Huang$^{1*}$, Jiayan Huang$^{1}$}\\
\footnotesize{$^1$ School of Mathematical Sciences, Huaqiao University, Quanzhou 362021, China}
}

\date{}

\begin{document}
\maketitle

%\keywords{simultaneous block diagonalization via congruence, center}

%\subjclass[2020]{15A69, 15A20, 13P05,15B57}

\blfootnote{$^*$ Corresponding author: hualin.huang@hqu.edu.cn}
%%%%%%%%%%%%
%%% ABSTRACT
%%%%%%%%%%%%
\vspace{-.2in}
\begin{abstract}
This article studies canonical forms derived from the finest simultaneous block diagonalization of a set of symmetric matrices via congruence. Our technique relies on the center theory of a set of multivariate polynomials, which we adapt for a set of symmetric matrices. We establish a bijective relationship between the simultaneous block diagonalizations of these matrices via congruence and complete sets of orthogonal idempotents of their centers. Based on this framework, we provide an algorithm relying primarily on standard linear algebra tasks. We also extend this technique to the simultaneous block diagonalizations of a set of Hermitian matrices via $*$-congruence. In addition, we investigate the simultaneous orthogonal block diagonalizations for real matrices and the simultaneous unitary block diagonalizations for Hermitian matrices. Several examples are provided to demonstrate its effectiveness.
\end{abstract}

\textbf{Keywords}: simultaneous diagonalization via congruence, center, symmetric matrices, Hermitian matrices

\textbf{MSC}: 15A69, 15A20, 13P05

\section{Introduction}\label{sec:intro}

\subsection{Problem statement}\label{sec:problem}

Let $\k$ be a field of characteristic $\ne 2$ and $A_1, A_2, \ldots, A_m \in \k^{n \times n}$ be a set of symmetric matrices. We are interested in determining the canonical forms of the finest simultaneous block diagonalization
\begin{align*}
	P^{T}A_iP&=\begin{pmatrix}
	   B_{i1}&&\\
    	&\ddots&\\
		&&B_{it}
	\end{pmatrix}
\end{align*}
via congruence for all $1 \le i \le m$, where $P\in \operatorname{GL}_n(\k)$ is an invertible matrix, $B_{ij}\in \k^{n_j \times n_j}$ and $\sum_{j=1}^{t}n_{j}=n$. This problem is called simultaneous block diagonalization via congruence (SBDC) of a set of symmetric matrices. If $t=n$, $A_1, A_2, \ldots, A_m$ are simultaneously diagonalized and it is referred to as simultaneous diagonalization via congruence (SDC). The SBDC of multiple symmetric matrices is a classical problem in linear algebra and has been of interest to many researchers. By reformulating the original problem into one with less complexity, large-scale problems may be solved more efficiently. Thus, the SBDC has found applications in many areas, including quadratic programming, variational analysis and signal processing, see~\cite{bustamante2020solving, jiang2016simultaneous, vollgraf2006quadratic} and references therein.

Let $q_1, q_2,\ldots,q_m \in \k[x_1,x_2,\ldots,x_n]$ be a set of quadratic forms defined by $q_i(x)=x^T A_i x$. It is of interest to find a change of variables $x=Py$ with $P\in \operatorname{GL}_n(\k)$ such that
\[ 
    q_i(Py)=q_{i1}(y_1, \dots, y_{a_1})+q_{i2}(y_{a_1+1}, \dots, y_{a_2}) + \dots + q_{it}(y_{a_{t-1}+1}, \dots, y_n)
\] 
for all $1 \le i \le m$. This is called the simultaneous direct sum decomposition or simultaneous diagonalization if $t=n$ for $q_1, q_2,\ldots,q_m.$ Due to the well-established equivalence between symmetric matrices and quadratic forms, the SDC and SBDC of a set of symmetric matrices correspond directly to the simultaneous diagonalization and direct sum decompositions of the associated set of quadratic forms. Specifically, the problem is equivalent to determining a change of variables $x=Py$ such that
\[
        q_i(Py) = (Py)^T A_i (Py) = y^T (P^{T}A_iP) y,
\]
where $P^{T}A_iP$ are block diagonal matrices, as defined above.

\subsection{Previous work}\label{sec:related}

As a special case, the SDC is less complex and has been studied thoroughly. Early work on the SDC of two real symmetric matrices was initiated by Weierstrass~\cite{weierstrass1868zur} in 1868, and pairs of quadratic forms were developed by authors like Finsler~\cite{finsler1936uber} and Calabi~\cite{calabi1964linear}. In 2007, Hiriart-Urruty~\cite{hiriart2007potpourri} posed an open question about the SDC of a set of real symmetric matrices. Until recently, Jiang and Li~\cite{jiang2016simultaneous} extended the SDC to more than two real symmetric matrices, which was further investigated by~\cite{nguyen2020simultaneous}. Some papers, such as~\cite{auyeung1970necessary, hong1986reduction, lancaster2005canonical}, considered the SDC of two Hermitian matrices, and it was expanded to a set of Hermitian matrices in~\cite{le2022simultaneous}. In more recent work~\cite{bustamante2020solving}, authors tackled the SDC for a set of general complex symmetric matrices. Additionally, some researchers explored the SDC using quadratic forms associated with symmetric matrices. The SDC of two quadratic forms has been approached in~\cite{becker1980necessary, hestenes1940theorem, wonenburger1966simultaneous}, and was extended to the almost SDC for a set of quadratic forms in~\cite{wang2024notion}.

Compared to the SDC, which poses strict constraints on the matrices, the SBDC concerns the canonical forms of general sets of symmetric matrices and is much more difficult to solve. Uhlig~\cite{uhlig1973simultaneous} generalized the SDC of two real symmetric matrices to the SBDC of two real symmetric matrices based on their Jordan normal forms. His later papers~\cite{uhlig1976canonical, uhlig1979recurring} further examined canonical pair forms of two real symmetric matrices and provided a new derivation. Until now, the SBDC of more than two symmetric matrices or Hermitian matrices has not been explored and remains a challenging problem. Furthermore, while canonical forms for a single matrix under restricted transformations like orthogonal congruence and unitary $*$-congruence are known~\cite{hong1989canonical, horn2009canonical}, extending the SBDC to a set of matrices under these transformations presents an even more complex challenge.

In \cite{fang2025simultaneous}, the authors extended Harrison's center theory \cite{harrison1975grothendieck} and applied it to the simultaneous direct sum decompositions of a general set of multivariate polynomials. Such investigation for a single homogeneous polynomial has been carried out in~\cite{huang2021diagonalizable, huang2021centres}. The problem of simultaneously decomposing a set of polynomials is connected to the simultaneous Waring decompositions of multiple polynomials~\cite{carlini2003waring, tiels2013coupled}. Authors~\cite{decuyper2019decoupling, dreesen2018decoupling, dreesen2015decoupling} then developed tensor decomposition-based techniques to simultaneously decompose several multivariate polynomials into sums of univariate polynomials. On the one hand, if we focus on a set of quadratic forms, then the idea and method of \cite{fang2025simultaneous} can be refined to the present SBDC problem. On the other hand, there is a surprising equivalence between the simultaneous direct sum decomposition of a general set of multivariate polynomials and the SBDC of a set of symmetric matrices. This is explained in the next subsection with necessary notations.

%one or two sentence to relate the SBDC to the simultaneous direct sum decomposition of multivarite polynomials: while they are developed for general cases, they can be extended specifically for quadratic forms. The direct sum decomposition of a homogeneous multivariate polynomial has been considered in \cite{buczynska2015apolarity, huang2021centres} via various approaches. The simultaneous direct sum decomposition is proposed as an open problem in \cite{dreesen2015decoupling} and explored in~\cite{dreesen2015block}, a technique is proposed in~\cite{fang2025simultaneous}. It is also related to the Waring problem of polynomials. The simultaneous Waring decomposition for a set of homogeneous polynomials was investigated in \cite{carlini2003waring, tiels2013coupled}. In the series of papers \cite{decuyper2019decoupling, dreesen2018decoupling, dreesen2015decoupling}, the authors developed a method to simultaneously decompose a set of nonhomogeneous multivariate real polynomials into linear combinations of univariate polynomials in linear forms of the input variables by tensor decompositions.

%Uhlig~\cite{uhlig1973simultaneous, uhlig1976canonical, uhlig1979recurring} 
%Uhlig~\cite{uhlig1973simultaneous} generalized the SDC of two real symmetric matrices to the SBDC of two real symmetric matrices. The results are based on the canonical form of a pair of real symmetric matrices. His approach also determines the number of blocks for the finest SBDC.
%Uhlig~\cite{uhlig1976canonical} derived canonical forms of a pair of real symmetric matrices.

\subsection{Methods and results}\label{sec:method}

It has been established in \cite{fang2025simultaneous} that a simultaneous direct sum decomposition of several forms corresponds to an additive decomposition of the unit of their center algebra into orthogonal idempotents. A criterion and an algorithm, mostly based on standard linear algebra tasks, were provided to simultaneously decompose arbitrary forms. The center of a set of multivariate polynomials $f_i(x_1, x_2, \dots, x_n) \in \k[x_1, x_2, \dots, x_n]$ for $1 \le i \le m$ is defined as
\begin{equation}\label{eqn:center_poly_multi}
    Z(f_1, f_2, \dots, f_m) := \{X \in \k^{n \times n} \mid (H_{f_i}X)^T=H_{f_i}X, \ 1 \le i \le m \},
\end{equation} 
where $H_{f_i}$ is the Hessian matrix of $f_i$. It was proved that $Z(f_1, f_2, \dots, f_m)$ is a special Jordan algebra, and it was applied to the simultaneous direct sum decompositions of $f_1, f_2, \ldots, f_m$. The main idea of this paper is to restrict the polynomials to a set of quadratic forms and apply the machinery of \cite{fang2025simultaneous} to the SBDC problem to derive more precise results.

As a critical foundation, the center of a set of quadratic forms defined in Equation~\eqref{eqn:center_poly_multi} is reformulated into that of a set of symmetric matrices. We prove a bijective relationship between the SBDC of a set of symmetric matrices and complete sets of orthogonal idempotent elements of their center. Furthermore, we also address simultaneous orthogonal block diagonalization for real symmetric matrices. A simple algorithm and several examples are provided. In addition, we generalize our approach to the SBDC of a set of Hermitian matrices, including the special case of simultaneous unitary block diagonalization. It is worthy of remarking that the center of a set of multivariate polynomials is essentially the center of a set of symmetric matrices, and thus direct sum decompositions of arbitrary polynomials amount to the SBDC problem of symmetric matrices. Indeed, if we write the Hessian matrix of $f_i$ as a sum of scalar symmetric matrices multiplied by suitable monomials, then the equation in \eqref{eqn:center_poly_multi} can be rewritten as the one in \eqref{eqn:center_multiple}, see Definition \ref{def:center_matrix_multi}. Therefore, the SBDC problem includes many important problems as special cases, and our approach provides a highly unified solution to all of them.

\subsection{Organization of the paper}\label{sec:organize}

The remainder of this paper is organized as follows.
In Section~\ref{sec:center_diaognal} we introduce and apply centers to the SBDC and orthogonal SBDC of a set of symmetric matrices, and give a simple algorithm.
Section~\ref{sec:example} presents more examples to demonstrate the effectiveness of this technique.
Section~\ref{sec:hermitian} extends the center theory to the SBDC and the unitary SBDC of a set of Hermitian matrices.

\section{The center and simultaneous block diagonalizations}\label{sec:center_diaognal}

%2. Center and block diagonalization of a set of symmetric matrices: centers, orthogonal decomposition of unit <=>simultaneous block diagonalization, theorem and algorithm.

\subsection{The center of a set of symmetric matrices}\label{sec:center_matrices}

We begin by defining the center of a set of symmetric matrices. Let $A \in \k^{n \times n}$ be a symmetric matrix and $q(x_1, x_2, \dots, x_n) \in \k[x_1, x_2, \dots, x_n]$ be its associated quadratic form. Since $q$ can be represented as $q(x)=x^TAx$, where $x = (x_1, x_2, \dots, x_n)^T$, one readily finds that $H_q=2A$. Thus, we can reformulate the center of a set of quadratic forms in Equation~\eqref{eqn:center_poly_multi} to the center of a set of symmetric matrices.

\begin{definition} \label{def:center_matrix_multi}
\emph{Let $A_1, A_2, \dots, A_m \in \k^{n \times n}$ be a set of symmetric matrices. The center of these matrices is defined as
\begin{equation}\label{eqn:center_multiple}
    Z(A_1, A_2, \dots, A_m) := \{X \in \k^{n \times n} \mid (A_iX)^T=A_iX, \ 1 \le i \le m \}.
\end{equation}
}
\end{definition}

\begin{remark}\label{rem:matrix_poly}
\emph{Under the assumption that the characteristic of $ \k$ is not equal to $2$, a set of symmetric matrices and their associated quadratic forms share the same center. For clarity, we focus on symmetric matrices instead of quadratic forms in the rest of this article. It is obvious that $Z(A_1, A_2, \dots, A_m)$ contains all scalar matrices.}
\end{remark}

\begin{example}\label{ex:example1}\emph{Consider the pair of symmetric matrices $A_1=\begin{pmatrix}-2&2&-2\\2&2&0\\-2&0&-1\end{pmatrix}$ and $A_2=\begin{pmatrix}5&7&-1\\7&5&1\\-1&1&-1\end{pmatrix}$. According to Equation~\eqref{eqn:center_multiple}, we compute the center of $A_1$ and $A_2$ as
\begin{equation*}
    Z(A_1, A_2)=\left\{\begin{pmatrix}a-5b-c-d+5e&-a-b-5d+e&a-2b-d-e\\6a+6b+c-6d-6e&6a&3b\\2c&12d&6e\end{pmatrix} \vline\, a, b, c, d, e \in\k \right\}.
\end{equation*}}
\end{example}

The center of a set of symmetric matrices may not have an associative algebraic structure as the center of a single higher degree form. However, it has a special Jordan algebraic structure as observed in \cite{fang2025simultaneous}. As a matter of fact, if $X, Y \in Z(A_1, A_2, \dots, A_m)$, then $X \odot Y := \frac{1}{2}(XY+YX) \in Z(A_1, A_2, \dots, A_m)$, hence $(Z(A_1, A_2, \dots, A_m), \odot)$ becomes a Jordan subalgebra of the full matrix Jordan algebra under the symmetric matrix product $\odot$. It is not hard to see that, a complete set of orthogonal idempotents of $(Z(A_1, A_2, \dots, A_m), \odot)$ is a set of mutually orthogonal matrices in $Z(A_1, A_2, \dots, A_m)$ which sum up to the identity matrix, see \cite{fang2025simultaneous} for more details.

\subsection{Simultaneous block diagonalization of symmetric matrices}\label{sec:simul_diagonal}

We call $\epsilon_1, \epsilon_2, \dots, \epsilon_t \in Z(A_1, A_2, \ldots, A_m)$ a complete set of orthogonal idempotent elements if \[ \epsilon_i^2=\epsilon_i \ \forall i, \quad \epsilon_i \epsilon_j=0 \ \forall i \ne j, \quad \sum_i \epsilon_i=I_n. \] It turns out that complete sets of orthogonal idempotent elements of centers are crux for the purpose of simultaneously block diagonalizing a set of symmetric matrices via congruence.

\begin{theorem}\label{thm:center}
Suppose $A_1, A_2, \ldots, A_m \in \k^{n \times n}$ are a set of symmetric matrices. Then
\begin{itemize}
    \item[{(1)}] There is a one-to-one correspondence between simultaneous block diagonalizations of $A_1, A_2, \ldots, A_m$ via congruence and complete sets of orthogonal idempotent elements of $Z(A_1, A_2, \ldots, A_m)$.
    \item[{(2)}] $A_1, A_2, \ldots, A_m$ are not simultaneously block diagonalizable via congruence if and only if $Z(A_1, A_2, \ldots, A_m)$ has no nontrivial idempotent elements.
    \item[{(3)}] $A_1, A_2, \ldots, A_m$ are simultaneously diagonalizable via congruence if and only if $Z(A_1, A_2, \ldots, A_m)$ has a complete set of $n$ orthogonal idempotent elements.
    \item[{(4)}] The finest simultaneous block diagonalization for $A_1, A_2, \ldots, A_m$ via congruence is unique up to equivalence.
\end{itemize}
\end{theorem}

\begin{proof}
A proof of item (1) is provided, which evidently implies items (2) and (3). Item (4) is a direct consequence of the uniqueness of the complete set of primitive orthogonal idempotent elements of unital Jordan algebras up to equivalence, see \cite{mccrimmon2003taste} for more details. 

We begin by examining how centers of a set of symmetric matrices vary under the congruence transformation. Let $B_1, B_2, \ldots, B_m \in \k^{n \times n}$ be a set of symmetric matrices. Suppose there exists an invertible matrix $P \in \GL_n(\k)$ such that $B_i=P^T A_i P$ for $1 \le i \le m$. According to Equation~\eqref{eqn:center_multiple}, we have $(B_iY)^T=B_iY$ for any $Y \in Z(B_1,B_2,\ldots, B_m)$. It follows that $(B_iY)^T = Y^{T}P^{T}A_iP=P^{T}A_iPY$, which shows that $PYP^{-1} \in Z(A_1,A_2,\ldots, A_m)$. Consequently, we have $Z(B_1,B_2,\ldots, B_m)=P^{-1}Z(A_1,A_2,\ldots, A_m)P.$ 

%(1)
Suppose symmetric matrices $A_1, A_2, \ldots, A_m$ are simultaneously block diagonalizable via congruence such that
\begin{align*}
    B_i = P^{T}A_iP =\begin{pmatrix}
					B_{i1}&&\\
					&\ddots&\\
					&&B_{it}
    \end{pmatrix}
\end{align*}
for all $1 \le i \le m$. For all $1 \le j \le m$, let
\[
    e_j=\begin{pmatrix}
    0 &&&& \\ & \ddots &&& \\ &&I_j&& \\ &&& \ddots & \\ &&&& 0
\end{pmatrix} 
\]
be block diagonal matrices which are conformal with the $B_i$'s. It is clear that the $e_j$'s are a complete set of orthogonal idempotent elements of $Z(B_1, B_2, \dots, B_m)$. Since $Z(A_1, A_2, \dots, A_m)=PZ(B_1, B_2, \dots, B_m)P^{-1}$, it is obvious that $Pe_1P^{-1}, Pe_2P^{-1}, \ldots,  Pe_tP^{-1}$ are a complete set of orthogonal idempotent elements of $Z(A_1, A_2, \ldots, A_m)$.

Conversely, suppose $\epsilon_1, \epsilon_2, \ldots, \epsilon_t$ in $Z(A_1, A_2, \ldots, A_m)$ are a complete set of orthogonal idempotent elements. Since the $\epsilon_i$'s are diagonalizable matrices and they pairwise commute, $\epsilon_1, \epsilon_2, \ldots ,\epsilon_t$ are simultaneously diagonalizable. Therefore, there exists an invertible matrix $P \in \GL_n(\k)$ such that $P^{-1}\epsilon_j P=e_j$ for all $1 \le j \le t$ as above. Let $B_i = P^{T}A_iP$ for all $1 \le i \le m$. We now have $B_ie_j=(B_ie_j)^T=e_j^TB_{i}^T=e_jB_i$ for $1\le i \le m$ and $1\le j \le t$. We further deduce that the $B_i$'s are block diagonal matrices and can be written as
			\begin{align*}
				B_i &=\begin{pmatrix}
					B_{i1}&&\\
					&\ddots&\\
					&&B_{it}
				\end{pmatrix}
			\end{align*}
for all $1 \le i \le m$. Therefore, a simultaneous block diagonalization of $A_1, A_2, \ldots, A_m$ via congruence corresponds to the complete set of orthogonal idempotent elements $\epsilon_1, \epsilon_2, \ldots, \epsilon_t$.
\end{proof}

For real symmetric matrices, it is more appropriate to consider simultaneous congruence via orthogonal matrices in some specific problems. Our approach can be easily adapted to this situation. 

\begin{corollary}[Orthogonal congruence]\label{cor:orthogonal}
Suppose $A_1, A_2, \ldots, A_m \in \R^{n \times n}$ are a set of real symmetric matrices. Then there is a one-to-one correspondence between simultaneous block diagonalizations of $A_1, A_2, \ldots, A_m$ via orthogonal congruence and complete sets of symmetric orthogonal idempotent elements of $Z(A_1, A_2, \ldots, A_m)$.
\end{corollary}

\begin{proof}
The correspondence for general cases was established in Theorem~\ref{thm:center}. We now refine it to the real orthogonal case. Suppose $Q\in \operatorname{O}_{n}(\R)$ is an invertible orthogonal matrix and $B_i = Q^TA_iQ$ are block diagonal as defined above. Then the corresponding idempotent elements $e_1,e_2,\dots,e_t$ of $Z(B_1, B_2, \ldots, B_m)$ are real, symmetric, and mutually orthogonal. Since $Q^{-1}=Q^{T}$, we have $Z(A_1, A_2, \ldots, A_m) = QZ(B_1, B_2, \ldots, B_m)Q^T$. Thus, $Q e_1 Q^T, Q e_2 Q^T,\ldots,$ $Q e_t Q^T$ form a complete set of symmetric orthogonal idempotent elements of $Z(A_1, A_2, \ldots, A_m)$.

Conversely, assume $Z(A_1,A_2\dots,A_m)$ contains a complete set of symmetric orthogonal idempotent elements $\epsilon_1,\epsilon_2,\dots,\epsilon_t$. Since $\epsilon_1,\epsilon_2,\dots,\epsilon_t$ are real, symmetric, and pairwise commute, they are simultaneously diagonalizable by an orthogonal matrix $Q\in \operatorname{O}_{n}(\R)$ such that $Q^T\epsilon_j Q=e_j$. The same argument as in Theorem~\ref{thm:center} shows that $Q^TA_iQ$ commutes with all $e_j$ and is therefore block diagonal. Hence $Q$ simultaneously block diagonalizes $A_1,A_2,\dots,A_m$ via orthogonal congruence. This completes the correspondence.    
\end{proof}

%$Qe_1Q^{-1}, Qe_2Q^{-1}$, $\ldots$, $ Qe_tQ^{-1}$

\subsection{Algorithm for simultaneous block diagonalization}\label{sec:diagonal_algorihm}

We describe the procedures for the SBDC of a set of symmetric matrices in Algorithm~\ref{alg:diagonal}. It is a recursive process that begins with computing the center of these symmetric matrices according to Definition~\ref{def:center_matrix_multi}. Using the base matrices of the center, a nontrivial idempotent~$\epsilon_1$ is found, and the second orthogonal idempotent is then defined as~$\epsilon_2=I-\epsilon_1$. If no orthogonal idempotent is found, the symmetric matrices are not simultaneously block diagonalizable via congruence. Otherwise, an invertible matrix $P$ is computed according to Theorem~\ref{thm:center} and these symmetric matrices are simultaneously block diagonalized via congruence as a set of block diagonal matrices, each with two diagonal block submatrices.

\begin{algorithm}
    \caption{Simultaneous block diagonalization of symmetric matrices}\label{alg:diagonal}
    \begin{algorithmic} [1]
    \Statex \textbf{Input}: a set of symmetric matrices~$A_1, A_2, \ldots, A_m$
    \Statex \textbf{Output}: two sets of symmetric matrices $B_{11}, B_{21},\ldots, B_{m1}$ and $B_{12}, B_{22}, \ldots, B_{m2}$
    \State Calculate the center~$Z(A_1, A_2,\ldots, A_m)$ by solving system $(A_iX)^T=A_iX$ for $1 \le i \le m$. Select base matrices $X_{j}$ for $1 \le j \le \dim{Z(A_1, A_2,\ldots, A_m)}$ such that $X=\sum_{j=1}^{\dim{Z(A_1, A_2,\ldots, A_m)}}X_j$. 
    \State Impose the constraint $X^2=X$ on the center $Z(A_1, A_2, \ldots, A_m)$ to compute a nontrivial idempotent~$\epsilon_1$. For orthogonal congruence, an additional constraint $X^T=X$ is imposed on $Z(A_1, A_2, \ldots, A_m)$. If $\epsilon_1$ does not exist, $A_1, A_2, \ldots, A_m$ are not simultaneously block diagonalizable via congruence. Otherwise, calculate the orthogonal idempotent $\epsilon_2=I-\epsilon_1$ and continue.
    \State Compute invertible matrix $P$ whose columns form a basis of simultaneous eigenvectors for $\epsilon_1$ and $\epsilon_2$, which guarantees $P^{-1}\epsilon_1P=\begin{pmatrix} I_{1} & \\  & 0\end{pmatrix}$ and $P^{-1}\epsilon_2P=\begin{pmatrix} 0 & \\  & I_{2}\end{pmatrix}$. For orthogonal congruence, take bases of column spaces of $\epsilon_1$ and $\epsilon_2$, then apply Gram–Schmidt orthogonalization to obtain an orthogonal matrix $P$.
    \State Simultaneously block diagonalize~$A_1, A_2, \ldots, A_m$ via congruence such that $P^TA_{i}P$ takes the form of a block diagonal matrix $\begin{pmatrix} B_{i1} & \\  & B_{i2}\end{pmatrix}$ for $1 \le i \le m$. Go to step 1 and continue to simultaneously block diagonalize $B_{11}, B_{21},\ldots, B_{m1}$ and $B_{12}, B_{22}, \ldots, B_{m2}$ via congruence separately.
    \end{algorithmic}
\end{algorithm}

Since these submatrices may admit further SBDC, Algorithm~\ref{alg:diagonal} is applied recursively. This process occurs recursively until all output matrices are not simultaneously block diagonalizable via congruence. The resulting output matrices can then be reassembled as block diagonal matrices corresponding to the original symmetric matrices~$A_1, A_2, \ldots, A_m$. The maximum number of blocks for this decomposition is uniquely determined by the structure of the primitive orthogonal idempotent elements in the associated Jordan algebra~\cite{mccrimmon2003taste}.

This algorithm contains several components, including computations of centers and orthogonal idempotents, which use standard tasks in linear algebra~\cite{huang2021centres}. Calculating the centers of symmetric matrices requires solving a system of linear equations, which has polynomial complexity~\cite{horn2012matrix}. The orthogonal idempotents are obtained by imposing constraints on the rank and trace of the matrix $X$ in Equation~\eqref{eqn:center_multiple}, which yields a system of nonlinear equations. For the SDC of a set of symmetric matrices, this system only contains sparse quadratic equations that are solved efficiently in numerical experiments~\cite{fang2025numerical}. For the general SBDC, determining idempotents may involve solving polynomial systems of higher complexity and remains a challenging problem. Nevertheless, for many practical instances with lower-dimensional centers, it remains computationally feasible.

\section{More examples}\label{sec:example}

This section presents additional examples to illustrate the procedure of this technique. While the theory is developed for a general field $\k$, the following computations are performed over either $\R$ or $\C$ for the convenience of the exposition. The SDC of two and three symmetric matrices are shown in Examples~\ref{exp:two_diag} and~\ref{exp:three_diag}, respectively. The SBDC and orthogonal SBDC of three symmetric matrices are given in Examples~\ref{exp:three_block} and~\ref{exp:orth_block_three}, respectively.

% - Example 1: two matrices, diagonalization
\begin{example}\label{exp:two_diag}
\emph{(Example~\ref{ex:example1} continued) Consider the following two real symmetric matrices
\begin{equation*}
    A_1=\begin{pmatrix}-2&2&-2\\2&2&0\\-2&0&-1\end{pmatrix}, \quad A_2=\begin{pmatrix}5&7&-1\\7&5&1\\-1&1&-1\end{pmatrix}.
\end{equation*}
Using the center calculated in Example~\ref{ex:example1}, we obtain the base matrices of $Z(A_1,A_2)$ as
\begin{equation*}
    X_1=\begin{pmatrix}1&-1&1\\6&6&0\\0&0&0\end{pmatrix}, \, X_2=\begin{pmatrix}-5&-1&-2\\6&0&3\\0&0&0\end{pmatrix}, \, X_3=\begin{pmatrix}-1&0&0\\1&0&0\\2&0&0\end{pmatrix},\, X_4=\begin{pmatrix}-1&-5&-1\\-6&0&0\\0&12&0\end{pmatrix}, \, X_5=\begin{pmatrix}5&1&-2\\-6&0&0\\0&0&6\end{pmatrix}.
\end{equation*}
For orthogonal congruence, we impose an additional symmetry condition $X^T=X$ on the idempotent $X\in Z(A_1,A_2)$. Since the only solutions are $0$ and $I_3$, $A_1$ and $A_2$ are not simultaneously block diagonalizable via orthogonal congruence. Thus, we proceed to calculate orthogonal idempotents for general congruence. Since $-X_3$ is a  primitive idempotent, we set $\epsilon_1=-X_3$ and $\epsilon_2=I-\epsilon_1$. A pair of orthogonal idempotents is produced as
\begin{equation*}
    \epsilon_1=\begin{pmatrix}1&0&0\\-1&0&0\\-2&0&0\end{pmatrix}, \quad \epsilon_2=\begin{pmatrix}0&0&0\\1&1&0\\2&0&1\end{pmatrix}
\end{equation*}
and we calculate an invertible matrix $P_1=\begin{pmatrix}-1&0&0\\1&1&0\\2&0&1\end{pmatrix}$ according to Theorem~\ref{thm:center} such that
\begin{equation*}
    P_1^{-1}\epsilon_1P_1=\begin{pmatrix}1&&\\&0&\\&&0\end{pmatrix}, \quad P_1^{-1}\epsilon_2P_1=\begin{pmatrix}0&&\\&1&\\&&1\end{pmatrix}.
\end{equation*}}

\emph{The matrices $A_1$ and $A_2$ can then be simultaneously block diagonalized via congruence as
\begin{equation*}
    P_1^TA_1P_1=\begin{pmatrix}
				0&&\\
				&2& 0\\
				&0&-1
			\end{pmatrix}, \quad P_1^TA_2P_1=\begin{pmatrix}
			0&&\\
			&5&1\\
			&1&-1
			\end{pmatrix}, 
\end{equation*}
respectively. Since neither matrix is fully diagonalized, we continue to simultaneously diagonalize their bottom-right submatrices, which are denoted as $B_1=\begin{pmatrix} 2&0\\0&-1\end{pmatrix}$ and $B_2=\begin{pmatrix}5&1\\1&-1\end{pmatrix}$.
Their center is calculated as
\begin{equation*}
    Z(B_1,B_2)=\left\{\begin{pmatrix}-3a+b&-a\\2a&b\end{pmatrix}\vline\, a, b\in\R\right\}. 
\end{equation*}
We obtain a new pair of idempotent elements using $Z(B_1,B_2)$ as
\begin{equation*}
    \epsilon_3=\begin{pmatrix}-1&-1\\2&2\end{pmatrix}, \quad \epsilon_4=\begin{pmatrix}2&1\\-2&-1\end{pmatrix} 
\end{equation*}
and derive $P_2=\begin{pmatrix}-1&-1\\2&1\end{pmatrix}$ such that $P_2^{-1}\epsilon_3 P_2=\begin{pmatrix}1&0\\0&0\end{pmatrix}$ and $P_2^{-1}\epsilon_4 P_2=\begin{pmatrix}0&0\\0&1\end{pmatrix}$. 
Thus, $B_1$ and $B_2$ are simultaneously diagonalized via congruence as 
\begin{equation*}
    P_2^TB_1P_2=\begin{pmatrix}-2&0\\0&1\end{pmatrix}, \quad P_2^TB_2P_2=\begin{pmatrix}-3&0\\0&2\end{pmatrix},
\end{equation*}
respectively.}

\emph{Finally, we can see that $A_1$ and $A_2$ are simultaneously diagonalized via congruence as 
\begin{equation*}
    P^TA_1P=\begin{pmatrix}0&&\\&-2&\\&&1\end{pmatrix}, \quad P^TA_2P=\begin{pmatrix}0&&\\&-3&\\&&2\end{pmatrix},
\end{equation*}
respectively, where
\begin{equation*}
    P=P_1\begin{pmatrix}
				1&0\\
				0&P_2
			\end{pmatrix}=
        \begin{pmatrix}
					-1&0&0\\
					1&-1&-1\\
					2&2&1
				\end{pmatrix}.
\end{equation*}
%Note that $1,1$-th entries of both transformed matrices are zero. This implies that $A_1$ and $A_2$ are degenerate matrices.
It is interesting to note that $A_1$ and $A_2$ are not simultaneously decomposable via orthogonal congruence, but they are simultaneously diagonalizable via general congruence. 
}
\end{example}

% - Example 2: three matrices, diagonalization
\begin{example}\label{exp:three_diag}
\emph{Consider the following symmetric matrices
\[
    A_1 = \begin{pmatrix} 1 & 0 & 0 \\ 0 & -1 & 1 \\ 0 & 1 & 2 \end{pmatrix}, \quad
    A_2 = \begin{pmatrix} 0 & 1 & -1 \\ 1 & -1 & 1 \\ -1 & 1 & 1 \end{pmatrix}, \quad
    A_3 = \begin{pmatrix} 2 & 1 & -1 \\ 1 & -3 & 3 \\ -1 & 3 & -2 \end{pmatrix}.
\]
%They are SDC over C, but only SBDC over R.
Firstly, we consider their SBDC over $\R$. We compute their center by Equation~\eqref{eqn:center_multiple} as
			\begin{equation*}
				Z(A_1,A_2,A_3)=\left\{\begin{pmatrix}
					a-b+c & -c & c \\ 
					c & a-b & b \\
					0 & 0 & a
				\end{pmatrix}\vline\, a,b\in \R \right\}
			\end{equation*}
and its base matrices are
			\begin{equation*}
				X_1=\begin{pmatrix}
					1 & 0 & 0\\
					0 & 1 & 0\\
					0 & 0 & 1
				\end{pmatrix},\quad 
				X_2=\begin{pmatrix}
					-1 & 0 & 0\\
					0 & -1 & 1\\
					0 & 0 & 0
				\end{pmatrix},\quad 
				X_3=\begin{pmatrix}
					1 & -1 & 1\\
					1 & 0 & 0\\
					0 & 0 & 0
				\end{pmatrix}.
			\end{equation*}
A pair of orthogonal idempotents is calculated as
			\begin{equation*}
				\epsilon_1=\begin{pmatrix}
					1 & 0 & 0\\
					0 & 1 & -1\\
					0 & 0 & 0
				\end{pmatrix},\quad 
				\epsilon_2=\begin{pmatrix}
					0 & 0 & 0\\
					0 & 0 & 1\\
					0 & 0 & 1
				\end{pmatrix} 
			\end{equation*}
and we have
\begin{equation*}
P_1^{-1}\epsilon_1P_1=\begin{pmatrix}1&&\\&1&\\&&0\end{pmatrix}, \quad P_1^{-1}\epsilon_2P_1=\begin{pmatrix}0&&\\&0&\\&&1\end{pmatrix},
\end{equation*}
where $P_1=\begin{pmatrix}
					1 & 1 & 0\\
					0 & 1 & 1 \\
					0 & 0 & 1
				\end{pmatrix}$.        
%Note that $X_1$ is a primitive idempotent element as $\rk X_1=1$ and $\tr(X_1)=1$. 
Thus, we can simultaneously block diagonalize the symmetric matrices via congruence as
			\begin{equation*}
				P_1^TA_1P_1=\begin{pmatrix}
					1 & 1 & \\
					1 & 0 &  \\
					 &  & 3
				\end{pmatrix},\quad 
				P_1^TA_2P_1=\begin{pmatrix}
					0 & 1 & \\
					1 & 1 & \\
					 &  & 2
				\end{pmatrix},\quad 
				P_1^TA_3P_1=\begin{pmatrix}
					2 & 3 &  \\
					3 & 1 &  \\
					 &  & 1
				\end{pmatrix}.
 			\end{equation*}}

\emph{We repeat this process using their top left submatrices, which are denoted as
				\begin{equation*}
				B_1=\begin{pmatrix}
					1 & 1\\
					1 & 0
				\end{pmatrix},\quad 
				B_2=\begin{pmatrix}
					0 & 1 \\
					1 & 1 
				\end{pmatrix},\quad 
				B_3=\begin{pmatrix}
					2 & 3 \\
					3 & 1 
				\end{pmatrix}.
			\end{equation*}	 
Their center is computed as
\[
    Z(B_1,B_2,B_3)=\left\{\begin{pmatrix} a-b & -b \\ b & a  \end{pmatrix}\vline\, a,b\in \R \right\}.
\]
We obverse that $Z(B_1,B_2,B_3)\cong \C$ as algebras over $\R$, in other words, it only contains trivial idempotent elements $0$ and $I_3$. This is also straightforward by computation. Therefore, $B_1,B_2,B_3$ are not simultaneously diagonalizable via congruence over $\R$ and $A_1,A_2,A_3$ are simultaneously block diagonalizable via congruence over $\R$ as above. }

\emph{Now we consider the SBDC problem over $\C$, then $Z(A_1, A_2, A_3)$ shares the same bases and can be simultaneously block diagonalized as above. In contrast, $Z(B_1,B_2,B_3)\cong \C\times\C$ over the complex field $\C$ and must contain nontrivial idempotents. We obtain a new pair of orthogonal idempotents $\epsilon_3 = \begin{pmatrix} \frac{3 - i\sqrt{3}}{6} & -\frac{i\sqrt{3}}{3} \\ \frac{i\sqrt{3}}{3} & \frac{3 + i\sqrt{3}}{6} \end{pmatrix}$ and $\epsilon_4 = \begin{pmatrix} \frac{3 + i\sqrt{3}}{6} & \frac{i\sqrt{3}}{3} \\ -\frac{i\sqrt{3}}{3} & \frac{3 - i\sqrt{3}}{6} \end{pmatrix}$, then derive a corresponding invertible matrix $P_2=\begin{pmatrix}\frac{-1 - i\sqrt{3}}{2} & \frac{-1 + i\sqrt{3}}{2} \\ 1 & 1 \end{pmatrix}$ such that $P_2^{-1}\epsilon_3P_2=\begin{pmatrix}1&0\\0&0\end{pmatrix}$ and $P_2^{-1}\epsilon_4P_2=\begin{pmatrix}0&0\\0&1\end{pmatrix}$. Thus, $B_1,B_2,B_3$ are simultaneously diagonalized via congruence as
\[
    P_2^T B_1 P_2=\begin{pmatrix} \frac{-3 - i\sqrt{3}}{2} &  \\  & \frac{-3 + i\sqrt{3}}{2} \end{pmatrix},\:\quad
    P_2^T B_2 P_2=\begin{pmatrix} -i\sqrt{3} &  \\  & i\sqrt{3} \end{pmatrix},\:\quad
    P_2^T B_3 P_2=\begin{pmatrix}-3 - 2i\sqrt{3} &  \\  &-3 + 2i\sqrt{3} \end{pmatrix},
\]
respectively. Finally, $A_1,A_2,A_3$ are simultaneously diagonalized via congruence over $\C$ as
\[
P^T A_1 P=\begin{pmatrix} \frac{-3 - i\sqrt{3}}{2} &  \\   & \frac{-3 + i\sqrt{3}}{2} & \\  &  & 3 \end{pmatrix},\quad
P^T A_2 P=\begin{pmatrix} -i\sqrt{3} &  &  \\  & i\sqrt{3} &  \\  &  & 2 \end{pmatrix},\quad
P^T A_3 P=\begin{pmatrix} -3 - 2i\sqrt{3}  &  &  \\  & -3 + 2i\sqrt{3}  &  \\  &  & 1 \end{pmatrix},
\]
respectively, where
\begin{equation*}
    P=P_1\begin{pmatrix}
				P_2&0\\
				0&1
			\end{pmatrix}= \begin{pmatrix}
    \frac{1 - i\sqrt{3}}{2} & \frac{1 + i\sqrt{3}}{2} & 0 \\
    1 & 1 & 1 \\
    0 & 0 & 1
\end{pmatrix}.
\end{equation*}}

\emph{This example indicates that the SBDC problem depends on the ground field.}
\end{example}

% - Example 3: three matrices, block diagonalization
\begin{example}\label{exp:three_block}
\emph{Suppose we have three complex symmetric matrices
			\begin{equation*}
				A_1=\begin{pmatrix}
					1&2&3\\
					2&8&16\\
					3&16&33
				\end{pmatrix},\quad 
				A_2=\begin{pmatrix}
					1&2&3\\
					2&6&12\\
					3&12&25
				\end{pmatrix},\quad 
				A_3=\begin{pmatrix}
					1&2&3\\
					2&7&16\\
					3&16&37
				\end{pmatrix}.
			\end{equation*}
We compute their center as
			\begin{equation*}
				Z(A_1,A_2,A_3)=\left\{\begin{pmatrix}
					a-8b+c&2a-b&3a\\
					0&-12b+c&-12b\\
					0&3b&c
				\end{pmatrix}\vline\, a,b,c\in \C \right\}
			\end{equation*}
with base matrices 
\begin{equation*}
    X_1=\begin{pmatrix}1&2&3\\0&0&0\\0&0&0\end{pmatrix}, \enspace 
    X_2=\begin{pmatrix}-8&-1&0\\0&-12&-12\\0&3&0\end{pmatrix}, \enspace X_3=\begin{pmatrix}1&0&0\\0&1&0\\0&0&1\end{pmatrix}.
\end{equation*}
Since $X_1$ is a primitive idempotent, we set $\epsilon_1=X_1$ and obtain a pair of orthogonal idempotents
\begin{equation*}
\epsilon_1=\begin{pmatrix}1&2&3\\0&0&0\\0&0&0\end{pmatrix}, \quad \epsilon_2=\begin{pmatrix}0&-2&-3\\0&1&0\\0&0&1\end{pmatrix}.
\end{equation*}}

\emph{Thus, we compute an invertible matrix $P=\begin{pmatrix}1&-2&-3\\0&1&0\\0&0&1\end{pmatrix}$ such that
\begin{equation*}
P^{-1}\epsilon_1P=\begin{pmatrix}1&&\\&0&\\&&0\end{pmatrix}, \quad P^{-1}\epsilon_2P=\begin{pmatrix}0&&\\&1&\\&&1\end{pmatrix}.
\end{equation*}
Then $A_1, A_2, A_3$ are simultaneously block diagonalized via congruence as
\begin{equation*}
    P^TA_1P=\begin{pmatrix}
				1&&\\
				&4&10\\
				&10&24
	\end{pmatrix}, \quad 
    P^TA_2P=\begin{pmatrix}
				1&&\\
				&2&6\\
				&6&16
	\end{pmatrix}, \quad 
    P^TA_3P=\begin{pmatrix}
			1&&\\
			&3&10\\
			&10&28
	\end{pmatrix},
\end{equation*}
respectively. We continue to simultaneously diagonalize their submatrices via congruence, which are represented as
\begin{equation*}    
    B_1=\begin{pmatrix}4&10\\10&24\end{pmatrix}, \quad B_2=\begin{pmatrix}2&6\\6&16\end{pmatrix}, \quad B_2=\begin{pmatrix}3&10\\10&28\end{pmatrix}.
\end{equation*}
Their center is calculated as
\begin{equation*}
   Z(B_1, B_2, B_3) = \left\{\begin{pmatrix}4a+b&4a\\-a&b\end{pmatrix}\vline\, a, b \in\C\right\}
\end{equation*}
and we obtain its base matrices as
\begin{equation*}    
    X_4=\begin{pmatrix}4&4\\-1&0\end{pmatrix}, \quad X_5=\begin{pmatrix}1&0\\0&1\end{pmatrix}.
\end{equation*}
One can verify by direct computation that there are no nontrivial idempotents in $Z(B_1, B_2, B_3)$. This also follows from the center's algebraic structure $Z(B_1, B_2, B_3)\cong \C[\varepsilon]/(\varepsilon^2)$. Thus, $B_1, B_2, B_3$ are not simultaneously diagonalizable via congruence, and $A_1, A_2, A_3$ are simultaneously block diagonalizable via congruence as shown above. }
\end{example}

% - Example 4: three matrices, orthogonal block diagonalization
\begin{example}\label{exp:orth_block_three}
\emph{Let real symmetric matrices be
\begin{equation*}
    A_1=\begin{pmatrix}
        -9 & 4 & 12 \\
        4 & 10 & 3 \\
        12 & 3 & -16
    \end{pmatrix}, \quad 
    A_2=\begin{pmatrix}
        16 & 8 & 12 \\
        8 & 5 & 6 \\
        12 & 6 & 9
    \end{pmatrix}, \quad 
    A_3=\begin{pmatrix}
        41 & -4 & 12 \\
        -4 & 20 & -3 \\
        12 & -3 & 34
    \end{pmatrix}.
\end{equation*}
We compute their center according to Equation~\eqref{eqn:center_multiple} as
\begin{equation*}
    Z(A_1, A_2, A_3) = \left\{ \begin{pmatrix} a+9b & 0 & -12b \\ 0 & a & 0 \\ -12b & 0 & a+16b \end{pmatrix} \vline\, a, b \in \R \right\}
\end{equation*}
and work out two base matrices
\begin{equation*}
    X_1 = \begin{pmatrix} 1 & 0 & 0 \\ 0 & 1 & 0 \\ 0 & 0 & 1 \end{pmatrix}, \quad 
    X_2 = \begin{pmatrix} 9 & 0 & -12 \\ 0 & 0 & 0 \\ -12 & 0 & 16 \end{pmatrix}.
\end{equation*}
To perform the orthogonal SBDC, we impose the symmetric condition on a non-trivial idempotent element from $Z(A_1, A_2, A_3)$. Since $X_2^2=25X_2$, we normalize $X_2$ and obtain a pair of orthogonal idempotents
\begin{equation*}
\epsilon_1=\begin{pmatrix} \frac{9}{25} & 0 & -\frac{12}{25} \\ 0 & 0 & 0 \\ -\frac{12}{25} & 0 & \frac{16}{25} \end{pmatrix}, \quad \epsilon_2 = \begin{pmatrix} \frac{16}{25} & 0 & \frac{12}{25} \\ 0 & 1 & 0 \\ \frac{12}{25} & 0 & \frac{9}{25} \end{pmatrix}.
\end{equation*}}

\emph{
Thus, we compute an orthogonal matrix $P = \frac{1}{5}\begin{pmatrix} 3 & 0 & 4 \\ 0 & 5 & 0 \\ -4 & 0 & 3 \end{pmatrix}$ such that
\begin{equation*}
P^{-1}\epsilon_1P=\begin{pmatrix}1&&\\&0&\\&&0\end{pmatrix}, \quad P^{-1}\epsilon_2P=\begin{pmatrix}0&&\\&1&\\&&1\end{pmatrix}.
\end{equation*}
By simultaneously block diagonalizing $A_1, A_2, A_3$ via orthogonal congruence, we obtain
\begin{equation*}
    P^T A_1 P = \begin{pmatrix} -25 &  &  \\  & 10 & 5 \\  & 5 & 0 \end{pmatrix}, \quad 
    P^T A_2 P = \begin{pmatrix} 0 &  &  \\  & 5 & 10 \\  & 10 & 25 \end{pmatrix}, \quad
    P^T A_3 P = \begin{pmatrix} 25 &  &  \\  & 20 & -5 \\  & -5 & 50 \end{pmatrix},   
\end{equation*}
which may be further simultaneously diagonalized via orthogonal congruence. We denote their bottom-right submatrices as
\begin{equation*}    
    B_1=\begin{pmatrix} 10 & 5 \\ 5 & 0\end{pmatrix}, \quad B_2=\begin{pmatrix} 5 & 10 \\ 10 & 25\end{pmatrix}, \quad B_2=\begin{pmatrix} 20 & -5 \\ -5 & 50\end{pmatrix},
\end{equation*}
respectively, and compute their center as
\begin{equation*}
   Z(B_1, B_2, B_3) = \left\{\begin{pmatrix} a & 0\\ 0& a\end{pmatrix}\vline\, a \in\R \right\}.
\end{equation*}
Since a nontrivial idempotent does not exist in $Z(B_1, B_2, B_3)\cong \R$, they are not simultaneously diagonalizable via congruence, and $A_1, A_2, A_3$ are simultaneously block diagonalizable via orthogonal congruence as shown above.}
\end{example}

\section{Simultaneous block diagonalization of Hermitian matrices}\label{sec:hermitian}

%- definitions of Hermitian matrices, properties
%- definitions of center for a set of Hermitian matrices

In this section, we extend our results on the SBDC of a set of symmetric matrices to the SBDC of a set of Hermitian matrices. We briefly recall the definition of Hermitian matrices, see~\cite{horn2012matrix} for more details.

\begin{definition} \label{def:hermitian}
\emph{A matrix $A \in \C^{n \times n}$ is Hermitian if $A^{*}=A$, where $A^{*}$ denotes the conjugate transpose of $A$.}
\end{definition}

Given a set of Hermitian matrices $A_1, A_2, \dots, A_m \in \C^{n \times n}$, we are considering the problem of finding the finest simultaneous block diagonalization via $*$-congruence
\begin{align*}
	P^{*}A_iP =\begin{pmatrix}
	   B_{i1}&&\\
    	&\ddots&\\
		&&B_{it}
	\end{pmatrix}
\end{align*}
for all $1 \le i \le m$, where $P \in GL_{n}(\C)$ is an invertible matrix, $B_{ij}\in \C^{n_j \times n_j}$ and $\sum_{j=1}^{t}n_{j}=n$. Note that the center of a set of symmetric matrices defined in Equation~\eqref{eqn:center_multiple} can be naturally extended to that of a set of Hermitian matrices.

%The SDC of Hermitian matrices has been studied in~\cite{hong1986reduction, le2022simultaneous}. Hong et al.~\cite{hong1986reduction} did some early work on the SDC of two Hermitian matrices and gave the necessary and sufficient conditions for such a unitary matrix. Thanh and Thi~\cite{le2022simultaneous} presented a criterion and an algorithm for the SDC of multiple Hermitian matrices. So far, no known work has addressed the simultaneous block diagonalization of Hermitian matrices by congruence. 

%Villacampa et al.~\cite{villacampa2019guided} 
%According to ..., we define the center of ... as ...

\begin{definition} \label{def:center_hermitian}
\emph{Let $A_1, A_2, \dots, A_m \in \C^{n \times n}$ be a set of Hermitian matrices. The center of these matrices is defined as
\begin{equation}\label{eqn:center_hermitian}
    Z(A_1, A_2, \dots, A_m) := \{X \in \C^{n \times n} \mid (A_iX)^{*}=A_iX, \ 1 \le i \le m \}.
\end{equation}
}
\end{definition}

%Discussions of the theoretical results given in the previous sections also apply to the center of Hermitian matrices.

It is clear that the center of a set of Hermitian matrices is a linear space over the real field $\R$ and contains all scalar matrices. Due to their similar algebraic structures, we may apply center $Z(A_1, A_2, \ldots, A_m)$ in Equation~\eqref{eqn:center_hermitian} to simultaneously block diagonalize a set of Hermitian matrices via $*$-congruence.  

\begin{theorem}\label{thm:hermitian}
Suppose $A_1, A_2, \ldots, A_m \in \C^{n \times n}$ are a set of Hermitian matrices and $Z(A_1, A_2, \ldots, A_m)$ is the center of these matrices. Then
\begin{itemize}
    \item[{(1)}] $(Z(A_1, A_2, \ldots, A_m), \odot)$ is a real special Jordan algebra.
    \item[{(2)}] There is a one-to-one correspondence between simultaneous block diagonalizations of $A_1, A_2, \ldots, A_m$ via (unitary) $*$-congruence and complete sets of (Hermitian) orthogonal idempotent elements of $Z(A_1, A_2, \ldots, A_m)$.
    \item[{(3)}] $A_1, A_2, \ldots, A_m$ are not simultaneously block diagonalizable via $*$-congruence if and only if $Z(A_1, A_2, \ldots, A_m)$ has no nontrivial idempotent elements.
    \item[{(4)}] $A_1, A_2, \ldots, A_m$ are simultaneously diagonalizable via $*$-congruence if and only if $Z(A_1, A_2, \ldots, A_m)$ has a complete set of $n$ orthogonal idempotent elements.
    \item[{(5)}] The finest simultaneous block diagonalization for $A_1, A_2, \ldots, A_m$ via $*$-congruence is unique up to equivalence.
\end{itemize}
\end{theorem}
\begin{proof}
Similar to Theorem \ref{thm:center} and Corollary \ref{cor:orthogonal}. 
\end{proof}

The SBDC of a set of Hermitian matrices follows the same procedure as the SBDC of a set of symmetric matrices. Thus, we adopt Algorithm~\ref{alg:diagonal} for the SBDC of a set of Hermitian matrices and provide two examples to illustrate the method. 
%Due to the similarity between Theorems~\ref{thm:center} and~\ref{thm:hermitian}, we adopt Algorithm~\ref{alg:diagonal} for the SBDC of a set of Hermitian matrices. Two examples are provided below to explain the process. 
Example~\ref{exp:herm_diag} shows the SDC of three Hermitian matrices, and Example~\ref{exp:herm_block} illustrates the SBDC of two Hermitian matrices. The unitary SBDC of three Hermitian matrices is given in Example~\ref{exp:unitary_block_three}.

%- Example 1: three matrices, diagonalization
\begin{example}\label{exp:herm_diag}
\emph{Consider the following Hermitian matrices
\begin{equation*}
    A_1=\begin{pmatrix}	1&1+i\\	1-i&1 \end{pmatrix}, \quad A_2=\begin{pmatrix}	2&2+2i\\2-2i&7 \end{pmatrix}, \quad A_3=\begin{pmatrix}	-2&-2-2i\\	-2+2i&1 \end{pmatrix}.
\end{equation*}
According to Equation~\eqref{eqn:center_hermitian}, their center is calculated as
\begin{equation*}
Z(A_1,A_2,A_3)=\left \{\begin{pmatrix}	a+b&a+bi\\	0&b \end{pmatrix} \vline\, a,b\in \mathbb{R}  \right \}.
\end{equation*}
For unitary $*$-congruence, we require $A_1,A_2,A_3$ to commute with each other. A direct calculation gives
\[
    A_1 A_2 = \begin{pmatrix} 4 & 9+9i \\ 8-8i & 8 \end{pmatrix}, \quad A_2 A_1 = \begin{pmatrix} 4 & 4+4i \\ 9-9i & 11 \end{pmatrix}.
\]
Since $A_1A_2\ne A_2A_1$, they do not commute and thus $A_1,A_2,A_3$ are not simultaneously diagonalizable via unitary $*$-congruence.
We then compute a pair of orthogonal idempotents for general congruence based on $Z(A_1,A_2,A_3)$ as
\begin{equation*}
    \epsilon_1=\begin{pmatrix}	0&-1-i\\	0&1 \end{pmatrix}, \quad \epsilon_2=\begin{pmatrix}	1&1+i\\	0&0 \end{pmatrix}.
\end{equation*}
}

\emph{According to Theorem~\ref{thm:hermitian}, there exists an invertible matrix $P$ such that
\begin{equation*}
    P^{-1}\epsilon_1P=\begin{pmatrix}1&0\\0&0 \end{pmatrix}, \quad P^{-1}\epsilon_2P=\begin{pmatrix}0&0\\0&1 \end{pmatrix},
\end{equation*}
where
\begin{equation*}
    P=\begin{pmatrix}	1+i&1\\	-1&0 \end{pmatrix}.
\end{equation*}
Thus, we can simultaneously diagonalize these Hermitian matrices via $*$-congruence as
\begin{align*}
    P^{*}A_1P &=\begin{pmatrix}	1+i&1\\	-1&0 \end{pmatrix}\begin{pmatrix}		1&1+i\\	1-i&1 \end{pmatrix} \begin{pmatrix}		1+i&1\\	-1&0 \end{pmatrix}=\begin{pmatrix}	-1&0\\	0&1 \end{pmatrix}, \\
    P^{*}A_2P &=\begin{pmatrix}	1+i&1\\	-1&0 \end{pmatrix}\begin{pmatrix}		2&2+2i\\2-2i&7 \end{pmatrix} \begin{pmatrix}		1+i&1\\	-1&0 \end{pmatrix}=\begin{pmatrix}	3&0\\	0&2 \end{pmatrix}, \\
    P^{*}A_3P &=\begin{pmatrix}	1+i&1\\	-1&0 \end{pmatrix}\begin{pmatrix}		-2&-2-2i\\	-2+2i&1 \end{pmatrix} \begin{pmatrix}		1+i&1\\	-1&0 \end{pmatrix}=\begin{pmatrix}	5&0\\	0&-2 \end{pmatrix}.
\end{align*}}
\end{example}

%- Example 2: two matrices, block diagonalization
\begin{example}\label{exp:herm_block}
\emph{Given two Hermitian matrices
\begin{equation*}
    A_1=\begin{pmatrix}	0&-2-2i&-3-2i\\	-2+2i&-2&4+6i\\-3+2i&4-6i&11 \end{pmatrix}, \quad A_2=\begin{pmatrix}	-2&-2-5i&-3-i\\	-2+5i&-3&11+6i\\-3+i&11-6i&19 \end{pmatrix}.
\end{equation*}
We compute their center as
\begin{equation*}
    Z(A_1,A_2)=\left \{\begin{pmatrix}	a+8b&6bi&-6bi\\	3bi&a-b&3b\\-2bi&2b&a \end{pmatrix} \vline\, a,b \in \R  \right \}.
\end{equation*}
By imposing the Hermitian condition on $Z(A_1,A_2)$, the only solution is $X=\lambda I$ for some $\lambda\in\C$. Thus, $A_1$ and $A_2$ are not simultaneously block diagonalizable via unitary $*$-congruence. We then obtain a pair of orthogonal idempotents
\begin{equation*}
    \epsilon_1=\begin{pmatrix}	6&6i&-6i\\	3i&-3&3\\-2i&2&-2 \end{pmatrix}, \quad 
    \epsilon_2=\begin{pmatrix}	-5&-6i&6i\\	-3i&4&-3\\2i&-2&3 \end{pmatrix}.
\end{equation*}
An invertible matrix $P_1$ is calculated such that
\begin{equation*}
    P_1^{-1}\epsilon_1P_1=\begin{pmatrix}1&&\\&0&\\&&0 \end{pmatrix}, \quad
    P_1^{-1}\epsilon_2P_1=\begin{pmatrix}0&&\\&1&\\&&1 \end{pmatrix},
\end{equation*}
where
\begin{equation*}
    P_1=\begin{pmatrix}		6&-3i&2i\\	3i&2&-1\\-2i&-1&1 \end{pmatrix}.
\end{equation*}
Now simultaneously block diagonalize $A_1$ and $A_2$ via $*$-congruence as
\begin{equation*}
    P_1^{*}A_1P_1=\begin{pmatrix}	2&&\\	&-1&1+i\\&1-i&1  \end{pmatrix}, \quad 
    P_1^{*}A_2P_1=\begin{pmatrix}	1&&\\	&-1&2+i\\&2-i&2  \end{pmatrix},
\end{equation*}
respectively.}

\emph{Next, see if the transformed matrices may be further simultaneously diagonalized via $*$-congruence. We denote their submatrices as
\begin{equation*}
    B_1 = \begin{pmatrix}		-1&1+i\\1-i&1  \end{pmatrix}, \quad B_2 = \begin{pmatrix}		-1&2+i\\2-i&2  \end{pmatrix},
\end{equation*}
respectively. The center of the submatrices is calculated as
\begin{equation*}
    Z(B_1,B_2)=\left \{\begin{pmatrix}	a+b-bi&-bi\\-b&a+bi \end{pmatrix} \vline\, a,b\in \R \right \}.
\end{equation*}
By direct computation, the existence of a nontrivial orthogonal idempotent in $Z(B_1, B_2)$ requires 
\begin{equation*}
    a=\dfrac{\sqrt{3}i+3}{6}, \ b=-\dfrac{\sqrt{3}i}{3} \quad \text{or} \quad a=\dfrac{-\sqrt{3}i+3}{6}, \ b=\dfrac{\sqrt{3}i}{3}.
\end{equation*}
Since such $a,b\notin \mathbb{R}$, a nontrivial idempotent 
%similar to $\begin{pmatrix}	1&0\\0&0 \end{pmatrix}$ 
does not exist. Indeed, this can also be seen by the fact that $Z(B_1,B_2)\cong \C$ as $\R$-algebras. Thus, $B_1,B_2$ are not simultaneously diagonalizable via $*$-congruence and $A_1,A_2$ are simultaneously block diagonalizable via $*$-congruence.}
\end{example}

% - Example 3: three Hermitian matrices, unitary block diagonalization
\begin{example}\label{exp:unitary_block_three}
\emph{Consider the following Hermitian matrices
\begin{equation*}
    A_1=\begin{pmatrix}
        9 & -6i & 3i \\
        6i & 9 & -6i \\
        -3i & 6i & 9
    \end{pmatrix}, \quad 
   A_2=\begin{pmatrix}
        -4 & -2 & 5 \\
        -2 & 8 & -2 \\
        5 & -2 & -4
    \end{pmatrix}, \quad 
    A_3=\begin{pmatrix}
        15 & -6 & 6 \\
        -6 & 6 & 12 \\
        6 & 12 & 6
    \end{pmatrix}.
\end{equation*}
First, we compute their center as
\begin{equation*}
    Z(A_1, A_2, A_3) = \left\{ \begin{pmatrix} a+4b & 2b & 4b \\ 2b & a+b & 2b \\ 4b & 2b & a+4b \end{pmatrix} \vline \, a, b \in \R \right\}
\end{equation*}
and obtain a pair of Hermitian orthogonal idempotents
\begin{equation*}
    \epsilon_1 = \frac{1}{9}\begin{pmatrix} 4 & 2 & 4 \\ 2 & 1 & 2 \\ 4 & 2 & 4 \end{pmatrix} \quad \epsilon_2 = \frac{1}{9}\begin{pmatrix} 5 & -2 & -4 \\ -2 & 8 & -2 \\ -4 & -2 & 5 \end{pmatrix}.
\end{equation*}
This yields an invertible matrix
    $P = \frac{1}{3}\begin{pmatrix}
        2 & 2 & i \\
        1 & -2 & 2i \\
        2 & -1 & -2i
    \end{pmatrix}$
such that
\begin{equation*}
    P_1^{-1}\epsilon_1P_1=\begin{pmatrix}1&&\\&0&\\&&0 \end{pmatrix}, \quad
    P_1^{-1}\epsilon_2P_1=\begin{pmatrix}0&&\\&1&\\&&1 \end{pmatrix}.
\end{equation*}
}

\emph{We can now simultaneously block diagonalize $A_1,A_2,A_3$ via unitary $*$-congruence as
\begin{equation*}
    P^* A_1 P = \begin{pmatrix}
        9 & 0 & 0 \\
        0 & 9 & 9 \\
        0 & 9 & 9
    \end{pmatrix}, \quad
    P^* A_2 P = \begin{pmatrix}
        0 & 0 & 0 \\
        0 & 0 & -9i \\
        0 & 9i & 0
    \end{pmatrix}, \quad
    P^* A_3 P = \begin{pmatrix}
        18 & 0 & 0 \\
        0 & 18 & 0 \\
        0 & 0 & -9
    \end{pmatrix}.
\end{equation*}
We then verify if their bottom right submatrices
\begin{equation*}
    B_{1} = \begin{pmatrix} 9 & 9 \\ 9 & 9 \end{pmatrix}, \quad
    B_{2} = \begin{pmatrix} 0 & -9i \\ 9i & 0 \end{pmatrix}, \quad
    B_{3} = \begin{pmatrix} 18 & 0 \\ 0 & -9 \end{pmatrix}
\end{equation*}
can be simultaneously diagonalized via $*$-congruence. Their center is calculated as
\begin{equation*}
   Z(B_1, B_2, B_3) = \left\{\begin{pmatrix} a & 0\\ 0& a\end{pmatrix}\vline\, a \in\R \right\},
\end{equation*}
which cannot produce nontrivial idempotent elements. Therefore, $B_1, B_2, B_3$ are not simultaneously diagonalizable via congruence, and the unitary SBDC of $A_1, A_2, A_3$ is the finest decomposition.}
\end{example}

\section*{Use of AI tools declaration}
The authors declare they have not used Artificial Intelligence (AI) tools in the creation of this article.

\section*{Acknowledgements}

Supported by Key Program of Natural Science Foundation of Fujian Province (Grant no. 2024J02018) and National Natural Science Foundation of China (Grant no. 12371037).

\section*{Conflict of interest}
The authors declare no conflicts of interest.

\begin{spacing}{.88}
\setlength{\bibsep}{2.pt}
\bibliographystyle{abbrvnat}
\bibliography{simultaneous_block}

\end{spacing}
\end{document}